\documentclass[psamsfonts,10pt]{amsart}
\usepackage{mathpazo}

\usepackage{overpic}
\usepackage{amsmath,amsthm,amssymb,amscd,amsfonts,amsbsy}
\usepackage{color}
\usepackage{hyperref,url}
\usepackage{float}
\usepackage{hyperref}
\usepackage{enumerate}
\usepackage{enumitem}   
\usepackage{mathrsfs}  
\usepackage{mathrsfs}  
\usepackage{amsthm}
\usepackage{amsmath}
\usepackage{amsfonts}
\usepackage{amssymb}
\usepackage{amsthm}
\usepackage{amsmath}
\usepackage{amsfonts}
\usepackage{amssymb}
\usepackage{mathtools}
\usepackage{mathrsfs}
\usepackage{comment}

\usepackage{amsmath, amssymb, graphics, setspace}
\usepackage[utf8]{inputenc}
\usepackage{listings,xcolor}

\newcommand{\mathsym}[1]{{}}
\newcommand{\unicode}[1]{{}}

\usepackage[normalem]{ulem}

\newcommand{\R}{\ensuremath{\mathbb{R}}}

\newcommand{\f}{\varphi}

\newcommand{\de}{\delta}

\newcommand{\sgn}{\mathrm{sign}}

\usepackage[toc,page]{appendix}
\usepackage{tikz-cd}

\newtheorem {theorem} {Theorem}

\newtheorem {proposition} [theorem]{Proposition}
\newtheorem {corollary}{Corollary}
\newtheorem {lemma}  [theorem]{Lemma}

\newtheorem {remark}{Remark}

\newtheorem {mtheorem} {Theorem}

\def\R{\mathbb R}

\title[On the non-existence of isochronous tangential centers]
{On the non-existence of isochronous tangential centers\\ in Filippov vector fields}
\author[D. D. Novaes and L. A. Silva]
{Douglas D. Novaes and Leandro A. Silva}

\address{Departamento de Matem\'{a}tica - Instituto de Matem\'{a}tica, Estat\'{i}stica e Computa\c{c}\~{a}o Cient\'{i}fica (IMECC) - Universidade
Estadual de Campinas (UNICAMP), \ Rua S\'{e}rgio Buarque de Holanda, 651, Cidade Universit\'{a}ria Zeferino Vaz, 13083-859, Campinas, SP,
Brazil}
\email{ddnovaes@unicamp.br}
\email{lasilva@ime.unicamp.br}

\allowdisplaybreaks
\begin{document}



\maketitle
\begin{abstract}
The isochronicity problem is a classical problem in the qualitative theory of planar vector fields  which consists in characterizing whether a center is isochronous or not, that is, if all the trajectories in a neighborhood of the center have the same period. This problem is usually investigated by means of the so-called period function.  In this paper, we are interested in exploring the isochronicity problem for tangential centers of planar Filippov vector fields.  By computing the period function for planar Filippov vector fields around tangential centers, we show that such centers are never isochronous. 
\end{abstract}

\section{Introduction}

In planar vector fields, the {\it isochronicity problem} concerns about distinguishing whether a center is isochronous or not.  Recall that a center of a planar vector field is called {\it isochronous} if the {\it period function}  $T:S\to\R$ is constant. The period function is defined in a Poincar\'{e} section $S$ transverse to the period annulus and corresponds to the period of the trajectory starting at a point in $S$. In other words, a center is isochronous provided that every trajectory in a neighborhood of it have the same period. The isochronicity problem goes back to C. Huygens with his studies on the pendulum clock that oscillates isochronously \cite{huygens}.  For smooth planar vector fields, Poincaré and Lyapunov showed that the isochronicity of a center is directly connected with its linearizability \cite{romanovski2009center}. Their discovery has driven the subsequent studies on the isochronicity problem, which has received a lot of attention since then (see, for instance, \cite{chavarriga99, pleshkan69, romanovski2007,sanchez}).

More recently, the isochronicity problem has also been considered for planar non-smooth vector fields of kind
\begin{equation} \label{sistemainicial0}
Z(x,y)= \begin{cases}
Z^+(x,y), & h(x,y) > 0, \\
Z^-(x,y),& h(x,y) < 0,
\end{cases}
\end{equation}
where $h:\R^2\to\R$ is a smooth function having $0$ as a regular value, $Z^{\pm}$ are smooth vector fields,  and  $\Sigma=h^{-1}(0)$ is the discontinuity manifold.  Here, we assume that the trajectories of \eqref{sistemainicial0} obey the Filippov's convention \cite{Filippov88}. Accordingly, \eqref{sistemainicial0} is called Filippov vector field. Regarding the existence of isochronous centers in Filippov vector fields, conditions on a family of piecewise quadratic systems were provided in \cite{coll00}  to ensure that the origin is an isochronous center. Equally important, one can find results about the non-existence of isochronous centers, for instance, in \cite{manosas05}, it was showed that the origin of the non-smooth oscillator $\ddot x+g(x)\,\sgn\,\dot x+x=0$ is never a isochronous center for any analytic function $g$ satisfying $g(0)=g'(0)=0$.  Finally, in \cite{buzzigassul13}, conditions were obtained for piecewise linear vector fields to have an isochronous center at the infinity.  The papers above, but the last one, investigated the isochronicity problem around a focus-focus center.

In \cite{novaessilva2020}, the center-focus and cyclicity problems were considered for planar Filippov vector fields with monodromic tangential singularities. It was obtained a general recursive formula for the Lyapunov coefficients, which control whether a monodromic singularity is a center or a focus. Here, using the ideas from  \cite{novaessilva2020}, we study the isochronicity problem for planar Filippov vector fields around tangential centers.

A tangential singularity $p\in\Sigma^t$ of a planar Filippov vector field $Z: U \in \mathbb{R}^2 \to \mathbb{R}^2$  is called a $(2k^+,2k^-)$-{\it monodromic tangential singularity} provided that $p$ is simultaneously an invisible $2k^{+}$-multiplicity contact of $Z^{+}$ with $\Sigma$ and an invisible $2k^{-}$-multiplicity contact of $Z^{-}$ with $\Sigma,$ and $Z$ has a first-return map defined on $\Sigma$ around $p$ (see \cite[Definition 1]{novaessilva2020}). Recall that, for a smooth vector field $F$, a {\it $k$-multiplicity contact} between $F$ and $\Sigma$ is defined as a point  $p\in\Sigma$ satisfying
\[
F h(p) = F^2h(p) = \ldots = F^{k-1}h(p) =0,\text{ and } F^{k} h(p)\neq 0,
\]
where $Fh(p)=\langle\nabla h(p),F(p)\rangle,$ for $n=1$, and $F^{n}h(p) =  F (F^ {n-1}h)(p),$ for $n>1$. In addition, we say that a $2k$-multiplicity contact is {\it invisible} for $Z^+$ (resp. $Z^-$) provided that $(Z^{+})^{2k}h(p)<0$ (resp. $(Z^{-})^{2k}h(p)>0$). 

A $(2k^+,2k^-)$-monodromic tangential singularity of a planar Filippov vector field is called $(2k^+,2k^-)$-{\it tangential center} provided that it has a neighborhood of where the first return map is the identity. 

\bigskip

The main result of this paper states that tangential centers are never isochronous.

\begin{mtheorem}\label{teo:iso}
A $(2k^+,2k^-)$-tangential center of a planar Filippov vector field is not isochronous. 
\end{mtheorem}

In order to prove Theorem \ref{teo:iso}, we shall construct the period function of planar Filippov vector fields around tangential centers. Accordingly, suppose that the Filippov vector field \eqref{sistemainicial0} has a $(2k^+,2k^-)$-tangential center at $p\in\Sigma.$ Assume that $p=(0,0)$ and $h(x,y)=y.$ It can be done, without loss of generality, by taking local coordinates. Thus, the Filippov vector field $Z,$  restricted to a neighborhood $V$ of the origin, writes as
\begin{equation} \label{sistemainicial}
Z(x,y)= \begin{cases}
Z^{+}(x,y)=(X^+(x,y), Y^+(x,y)), & y > 0, \\
Z^{-}(x,y)=(X^-(x,y), Y^-(x,y)),& y < 0,
\end{cases}
\end{equation}
 for which the discontinuity manifold is now given by $\Sigma=\{(x,y)\in V:\,y=0\}$.
It is stated in \cite{novaessilva2020} that the origin is a $(2k^+,2k^-)$-monodromic tangential singularity for \eqref{sistemainicial} if, and only if, the following conditions hold:

\bigskip

\noindent{\bf C1.} $X^{\pm}(0,0)\neq0,$ $Y^{\pm}(0,0)=0,$ $\dfrac{\partial^i Y^{\pm}}{\partial x^i}(0,0)=0$  for  $i=1,\ldots,2k^{\pm}-2,$ and  $\dfrac{\partial^{2k^{\pm}-1} Y^{\pm}}{\partial x^{2k^{\pm}-1} }(0,0)\neq0$;

\medskip

\noindent{\bf C2.} $X^+(0,0)\dfrac{\partial^{2k^{+}-1} Y^{+}}{\partial x^{2k^{+}-1} }(0,0)<0$ and $X^-(0,0)\dfrac{\partial^{2k^{-}-1} Y^{-}}{\partial x^{2k^{-}-1} }(0,0)>0;$

\bigskip

\noindent{\bf C3.} $X^+(0,0) X^-(0,0)<0.$

\bigskip

\noindent Summarizing, conditions {\bf C1} and {\bf C2} provide that the origin is an invisible $2k^+$-multiplicity contact (resp.  invisible  $2k^-$-multiplicity contact)  between $Z^+$ (resp. $Z^-$) and $\Sigma,$ and condition {\bf C3} provides that the trajectories of $Z^+$ and $Z^-$ can be concatenated  at $\Sigma$ around the origin in order to define a first-return map.

Accordingly, half-return maps $\varphi^{+}$ and $\varphi^{-}$ on $\Sigma$ around $0$ are defined by the flows of $Z^{+}|_{y\geq0}$ and $Z^{-}|_{y\leq0},$ respectively. Since we are assuming that the Filippov vector field \eqref{sistemainicial} has a tangential center at the origin, then $\varphi^+=\varphi^-$ and, in this case, we denote the half-return map only by $\f$. In addition, $\f$ is a smooth involution around the origin satisfying $\f(0)=0$ (see \cite[Section 4.2]{AndGomNov19}) and, therefore, $\f'(0)=-1$. 

Now, since
$X^{\pm}(0,0)\neq0$,  there exists a small neighborhood $U$ of the origin such that $X^{\pm}(x,y) \neq0$ for all $(x,y) \in U.$  Thus, we can define 
	\begin{equation}\label{eq:eta} 
		\eta^+(x,y) = \delta\dfrac{Y^+(x,y)}{X^+(x,y)}\,\text{ and }\, \eta^-(x,y) = -\delta\dfrac{Y^-(x,y)}{X^-(x,y)},
	\end{equation}
	where
	\begin{equation}\label{delta}
\delta=\textrm{sign}(X^+(0,0))=-\textrm{sign}(X^-(0,0)).
\end{equation}
Notice that $\delta>0$ (resp. $\delta<0$) implies that the flow of $Z$ turns around the origin in the clockwise (resp. anti-clockwise) direction.

Finally, for each $x\in\R$ such that $(x,0)\in U$, let $t\mapsto y^{\pm}(t,x)$ be the solution of the initial value problem  
\begin{equation}\label{ivp}
\dfrac{dy}{dt}=\eta^{\pm}(x,y),\quad y(0)=0.
\end{equation}
	
The next result provides an expression for the period function of the Filippov vector field \eqref{sistemainicial} around the tangential center at the origin in terms of the half-return map $\f$ and the functions $y^{\pm}$.

\begin{mtheorem}\label{teo:period}
	Assume that the Filippov vector field \eqref{sistemainicial} has a $(2k^{+},2k^-)$-tangential center at the origin and let $\f$ be the associated half-return map. Then, the period function is given by
\begin{equation}\label{pf}
	T(x)= \delta (T^-(x)-T^+(x)),
\end{equation}	
where 
\begin{equation}\label{TpTn}
	\begin{aligned}
	T^{\pm}(x)&= (\f(x)-x) \int_0^1 \dfrac{1}{X^{\pm}\big(x+(\f(x)-x) t, y^{\pm}(\pm \delta (\f(x)-x) t, x)\big)}\,dt.
	\end{aligned}
\end{equation}	
\end{mtheorem}

Theorem \ref{teo:period} is proven in Section \ref{proof:teoB}. In its proof, we shall see that $T^-(x)T^+(x)\leq0$ and that the constant $\de$ corrects the sign of  $T^-(x)-T^+(x)$ in such way that $T(x)\geq0$.

It  has been proven in  \cite[Theorem A]{novaessilva2020}  that the half-return map $\f$ is analytic in a neighborhood of $x=0$ provided that $Z^{+}$ and $Z^{-}$ are  analytic in a neighborhood of the origin. Therefore, one can easily see that the period function $T(x)$ given by Theorem \ref{teo:period} is also analytic in a neighborhood of $x=0$ provided that $Z^{+}$ and $Z^{-}$  are analytic in a neighborhood of the origin. In this case, the {\it period constants} $\widehat{T}_i$, $i\in\mathbb{N}$, are defined as the coefficients of the power series of $T(x)$ around $x=0$, that is,
\begin{equation}\label{def:pc}
	T(x)= \sum_{i=0}^{\infty} \widehat{T}_i x^i.
\end{equation}
An Appendix is provided with the formulae for computing all the period constants.

Clearly, an isochronous center must satisfies $\widehat T_0\neq0$ and $\widehat T_i=0$ for every $i\in\mathbb{N}\setminus\{0\}$. Accordingly, the proof of Theorem \ref{teo:iso} follows immediately from the following corollary of Theorem \ref{teo:period}, which can be obtained from \eqref{pf} and \eqref{TpTn} just by noticing that $\f'(0)=-1$ and taking into account condition {\bf C3} and relationship \eqref{delta}. 

\begin{corollary}\label{cor:noisochronous} Assume that Filippov vector field \eqref{sistemainicial} has a $(2k^{+},2k^-)$-tangential center at the origin and let $T(x)$ be the period function given by \eqref{pf}. Then,
\[
\widehat{T}_0:=T(0)=0\,\text{ and }\, \widehat{T}_1:=T'(0)= 2\delta\Big(\dfrac{X^-(0,0)-X^+(0,0)}{X^+(0,0)X^-(0,0)}\Big)>0.
\]
\end{corollary}

\begin{remark}
Another problem related to the isochronicity problem is the {\it criticality problem}, which was introduced by Chicone and Jacobs \cite{chicone89} and  concerns about the bifurcation of critical periods of perturbed centers. A {\it critical period}  is defined as a critical point of the period function, that is, a  point $\rho > 0$ satisfying $T'(\rho) =0$. In addition, it is called {\it simple} provided that $T''(\rho) \neq 0$. 
 The number of critical periods and the number of simple critical periods provide, respectively, an upper and a lower bound for the number of oscillations of the period function. As a consequence of Corollary \ref{cor:noisochronous}, a perturbed tangential center does not admit critical periods in a neighborhood of $x=0$  and, therefore, the period function does not oscillate in this neighborhood.
\end{remark}

\section{Proof of Theorem \ref{teo:period}}\label{sec:period}

This section is devoted to the proof of Theorem \ref{teo:period}. It is based on a time-reparametrization of Filippov vector fields around a $(2k^{+}, 2k^{-})$-monodromic tangential singularity, for which the period function can be easily computed. Thus, the following result will be of major importance for recovering the period function of the original Filippov vector field:
\begin{proposition} [{\cite[Proposition 1.14]{chicone}}]\label{prop:criticality}
Let $U \in \mathbb{R}^n$ be an open set, $F: U \to \mathbb{R}^n$  a smooth vector field, and $g: U \to \mathbb{R}$ a positive smooth function. Consider the following differential equations
\begin{equation}\label{eq1:prop}
	\dot{x}=F(x)
\end{equation}
and 
\begin{equation}\label{eq2:prop}
	\dot{x}=g(x)F(x).
\end{equation}
If $J \subset \mathbb{R}$ is an open interval containing the origin and $\gamma: J \to \mathbb{R}^n$ is a solution of the differential equation \eqref{eq1:prop} with $\gamma (0)=x_0 \in U$, then the function $B: J \to \mathbb{R}$ given by 
\[
	B(t)= \int_{0}^{t} \dfrac{1}{g(\gamma(s))}ds
\] is invertible on its range $K \subset \mathbb{R}$. If $\rho: K \to J$ denotes the inverse of $B$, then the identity
\[
	\rho'(t)=g(\gamma(\rho(t)))
\]
	 holds for all $t \in K$ and the function $\sigma: K \to \mathbb{R}^n$ given by $\sigma(t)=\gamma(\rho(t))$ is the solution of the differential equation \eqref{eq2:prop} with initial condition $\sigma(0)=x_0.$
\end{proposition}

\subsection{Reparametrization of time}\label{sec:cf}

As commented above, assuming that the Filippov vector field \eqref{sistemainicial} has a $(2k^{+}, 2k^{-})$-tangential center at the origin (see conditions {\bf C1}, {\bf C2}, and {\bf C3}), there exists a small neighborhood $U$ of the origin such that $X^{\pm}(x,y) \neq0$ for all $(x,y) \in U.$  Taking into account that $|X^{\pm}(x,y)| = \pm \delta X^{\pm}(x,y)$ for every $(x,y)\in U,$ a reparametrization of time can be performed in order to transform the Filippov vector field \eqref{sistemainicial} restricted to $U$ into \begin{equation} \label{sistemacanonico} 
		\widetilde{Z}(x,y)=\begin{cases}
			\widetilde{Z}^+(x,y)=(\delta, \eta^+(x,y)), &y > 0, \\
			\widetilde{Z}^-(x,y)=(-\delta, \eta^-(x,y)), &y < 0,
		\end{cases}
	\end{equation}
	where  $ \eta^{\pm}$ and $\delta$ are given by  \eqref{eq:eta} and \eqref{delta}, respectively. Notice that
	\begin{equation}\label{eq:rep}
	\widetilde{Z}^{\pm}(x,y)=\dfrac{Z^{\pm}(x,y)}{|X^{\pm}(x,y)|}.
	\end{equation}
It is worth mentioning that the time-reparametrization above was the first step  in \cite{novaesrondon2020}  for the obtention of a canonical form for Filippov vector fields around a $(2k^{+}, 2k^{-})$-monodromic tangential singularity. Afterward, such a canonical form was used in \cite{novaessilva2020} for computing the Lyapunov coefficients of $(2k^{+}, 2k^{-})$-monodromic tangential singularities.

\begin{lemma}
Assume that the Filippov vector field \eqref{sistemainicial} has a $(2k^{+},2k^-)$-tangential center at the origin and let $\f$ be the associated half-return map. Then, the period function of its time-reparametrization \eqref{sistemacanonico} is given by $\widetilde{T}(x)= 2(x-\f(x))$.
\end{lemma}
\begin{proof}
In order to compute the period function of \eqref{sistemacanonico}, we shall consider the half-period functions $\widetilde{T}^+$ and $\widetilde{T}^-$ defined, respectively, as the flight-times taken for the trajectories of $\widetilde{Z}^+$ and $\widetilde{Z}^-,$ starting at $(x, 0) \in \Sigma\cap U$, for $x\geq0$, to reach $\Sigma$ again. Accordingly, for $x\geq0$ small, the period function $\widetilde{T}$ of \eqref{sistemacanonico} is defined as $\widetilde{T}(x)= \delta (\widetilde{T}^-(x)-\widetilde{T}^+(x))$. Notice that $\widetilde T(x)\geq 0$. Indeed, for $\delta>0$ (resp. $\delta<0$) one has that the flow of $Z$ turns around the origin in the clockwise (resp. anti-clockwise) direction and, therefore, 
$\widetilde T^+(x)\leq 0\leq \widetilde T^-(x)$ (resp. $\widetilde T^-(x)\leq 0\leq \widetilde T^+(x)$) for every $x\geq0$ such that $(x, 0) \in \Sigma\cap U$.

Now, the trajectories of $\widetilde Z^{\pm}$ with initial condition $(x,0)\in \Sigma\cap U$ are given by
\begin{equation} \label{solution}
\gamma^{\pm}(t,x) = (x \pm \delta t, y^{\pm}(t,x)),
\end{equation}
where $t\mapsto y^{\pm}(t,x)$ is the solution of the initial value problem \eqref{ivp}.

 	Consider the transversal sections $\Sigma^{\perp}_+= \{(x,y) \in U: x=0,\, y>0 \}$ and $\Sigma_-^{\perp}= \{(x,y) \in U: x=0,\, y<0 \}$ (see Fig. \ref{mapadobradobra}). It is clear that the flight-times taken for the trajectories of $Z^{+}$, starting at the points $(x,0)$ and $(\f(x),0),$ to reach  the section $\Sigma^{\perp}_{+}$ are given by $t_1^{+}=-\delta x$ and $t_2^{+}=-\delta\f(x)$, respectively. Analogously, the flight-times taken for the trajectories of $Z^{-}$, starting at the points $(x,0)$ and at $(\f(x),0),$ to reach the section $\Sigma^{\perp}_{-}$ are given by $t_1^{-}=\delta x$ and $t_2^{-}=\delta\f(x)$, respectively. Therefore, the half-period functions are given by $\widetilde{T}^{\pm}(x)=t_1^{\pm}-t_2^{\pm}=\pm\delta(\f(x)-x)$, which yields $\widetilde{T}(x)=\delta (\widetilde{T}^-(x)-\widetilde{T}^+(x))=2(x-\f(x))$. \end{proof}

 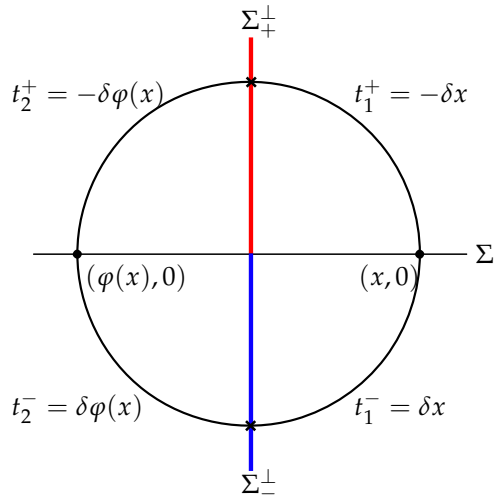
\begin{figure}[H]
 \bigskip
	\centering
	\begin{overpic}[scale=2.18]{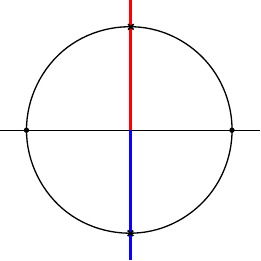}
		\put(102,48){$\Sigma$}
		\put(48,102){$\Sigma^{\perp}_+$}
		\put(48,-5){$\Sigma^{\perp}_-$}
		\put(-5,85){$t_2^+=-\delta\f(x)$}
		\put(74,85){$t_1^+=-\delta x $}
		\put(-5,13){$t_2^-= \delta \f(x)$}
		\put(74,13){$t_1^-= \delta x $}
		\put(12,43){$(\f(x),0)$}
		\put(75,43){$(x,0)$}
	\end{overpic}
	\vspace{0.8cm}
	\caption{Transversal sections $\Sigma^{\perp}_{-}$ and $\Sigma^{\perp}_{+}$ and the flight-times taken for the trajectories of $\widetilde Z^+$ and $\widetilde Z^-$, starting at the points $(x,0)$ and $(\f(x),0),$ to reach $\Sigma^{\perp}_{-}$ and $\Sigma^{\perp}_{+}$, respectively.}
	\label{mapadobradobra}
\end{figure}

\subsection{Proof of Theorem \ref{teo:period}}\label{proof:teoB}

Now, Lemma \ref{sec:period} can be applied together with Proposition \ref{prop:criticality} in order to compute the period function of the Filippov vector field \eqref{sistemainicial} by assuming that it has a $(2k^+, 2k^-)$-tangential center.

Analogously to the proof of Lemma \ref{sec:period}, the period function $T$ of \eqref{sistemainicial} is given by
\[ 
	T(x)= \delta (T^-(x)-T^+(x))\geq 0,
\]
where $T^{\pm}$ are the half-period functions of $Z^{\pm}$.

From \eqref{eq:rep}, $Z^{\pm}|_{U}=g^{\pm}(x,y)\widetilde{Z}^{\pm}$, where $g^{\pm}(x,y)=|X^{\pm}(x,y)|= \pm\delta X^{\pm}(x,y)$. Hence, if $\sigma^{\pm}(t,x)$ denotes the trajectory of $Z^{\pm}$ with initial condition $(x,0)\in\Sigma\cap U$, then, from Proposition \ref{prop:criticality}, 
\begin{equation}\label{sigma}
\sigma^{\pm}(t,x)=\gamma^{\pm}(\rho^{\pm}_x(t),x),
\end{equation}
where $\gamma^{\pm}(t,x),$ given by \eqref{solution}, is the trajectory of $\widetilde Z^{\pm}$ with initial condition $(x,0)\in\Sigma\cap U$, $\rho^{\pm}_{x}=(B^{\pm}_{x})^{-1}$, and
$$B_{x}^{\pm}(\tau)=\int_{0}^{\tau} \dfrac{1}{g^{\pm}(\gamma^{\pm}(s,x))}ds=\int_0^{\tau} \dfrac{\pm \delta}{X^{\pm}(x\pm \delta s, y^{\pm}(s, x))} ds.$$ 

Thus, from \eqref{sigma},
\[
 \gamma^{\pm}(\rho^{\pm}_{x}(T^{\pm}(x)),x)=\sigma^{\pm}(T^{\pm}(x),x) \in \Sigma\cap U.
\]
Consequently, from the characterization of the half-period functions $\widetilde T^{\pm}$ of \eqref{sistemacanonico}, we have that 
\begin{equation*}
	\rho_{x}^{\pm}(T^{\pm}(x))= {\widetilde{T}}^{\pm}(x) = \pm \delta(\f(x)- x),
\end{equation*} 
which implies that
\begin{equation*}
	T^{\pm}(x) =B_{x}^{\pm}(\pm\delta(\f(x)- x)).
\end{equation*}
Hence,
\begin{equation*}
	\begin{aligned}
	 T^{\pm}(x)= \int^{ \pm\delta(\f(x)- x)}_0 \dfrac{\pm \delta}{X^{\pm}(x\pm \delta s, y^{\pm}(s, x))}ds.
	\end{aligned}
\end{equation*}

We conclude this proof by performing the change of variables $s= \pm \delta(\f(x)-x)t$ in the integrals above, which yields
\[
T^{\pm}(x)=(\f(x)-x) \int_0^1 \dfrac{1}{X^{\pm}(x+(\f(x)-x) t, y^{\pm}(\pm \delta (\f(x)-x) t, x))} dt.
\]

\section*{Appendix: Period Constants}
This appendix provides the formulae for the computation of the period constants.
Notice that, from \eqref{def:pc} and \eqref{pf},  
\[
\widehat{T}_i = \dfrac{1}{i!}T^{(i)}(0)=\dfrac{1}{i!}\delta\big((T^-)^{(i)}(0)-(T^+)^{(i)}(0)\big).
\]
 In addition, from \eqref{TpTn},
\[
(T^{\pm})^{(i)}(0) = \int_0^{1} \dfrac{\partial^i}{\partial x^i} \left(\dfrac{\f(x)-x}{X^{\pm}(x+(\f(x)-x) t, y^{\pm}(\pm \delta (\f(x)-x) t, x))} \right) \Bigg|_{x=0}\,ds.
\]

Thus, in order to compute $\widehat{T}_i,$ it only remains to know how to compute the higher derivatives of the functions $x\mapsto y^{\pm}(\pm \delta (\f(x)-x) t, x))$ and $\f(x)$ at $x=0$. This has been done in \cite{novaessilva2020} as follows:

 Let
\begin{equation}\label{eq:seriesy}
	y^{\pm}(t, x) = \sum_{i=1}^{\infty} \dfrac{y_i^{\pm}(x)}{i!}t^i
\end{equation}
 and 
\begin{equation}\label{eq:series}
	\begin{aligned}
		{\varphi}(x) = -x + \sum_{n=2}^{\infty}\alpha_n x^n.
	\end{aligned}
\end{equation}
 Recall that ${\varphi}'(0)=-1$.

The coefficient functions $y_i$  of the series \eqref{eq:seriesy} are given recursively by
\[
	\begin{aligned}
		y_1^{\pm} (x)=& a^{\pm} x^{2k^{\pm}-1}+x^{2k^{\pm}}f^{\pm}(x),&\\
		y_i^{\pm} (x) = & (\pm \delta)^{i-1}\left( a^{\pm} \dfrac{(2k^{\pm}-1)!}{(2k^{\pm}-i)!}x^{2k^{\pm}-i}  +\sum_{l=0}^{i-1} {{i-1}\choose{l}}  \dfrac{(2k^{\pm})!}{(2k^{\pm}-l)!}x^{2k^{\pm}-l} {f^{\pm}}^{(i-1-l)} (x)\right) & \\  & +\sum_{l=1}^{i-1} \sum_{j=1}^{l}j {i-1\choose l} (\pm \delta)^{i-l-1}B_{l,j}(y^{\pm}_1(x),\dots , y^{\pm}_{l-j+1}(x))  \dfrac{\partial^{j+i-l-2}g^{\pm}}{\partial x^{i-l-1} \partial y^{j-1}}(x,0), \text{ if } 2\leq i\leq 2k^{\pm},& \\
		y_i^{\pm}(x) = & (\pm\delta)^{i-1} \Bigg(  {{i-1}\choose{2k^{\pm}}}(2k^{\pm})!{f^{\pm}}^{i-1-2k^{\pm}}(x)   +\sum_{l=0}^{2k^{\pm}-1} {{i-1}\choose{l}}  \dfrac{(2k^{\pm})!}{(2k^{\pm}-l)!}x^{2k^{\pm}-l}{f^{\pm}}^{(i-l-1)}(x)\Bigg) &  \\  &+ \sum_{l=1}^{i-1}\sum_{j=1}^{l}j {{i-1}\choose{l}} (\pm \delta)^{i-l-1}B_{l,j}(y_1^{\pm}(x),\dots , y^{\pm}_{l-j+1}(x)) \dfrac{\partial^{j+i-l-2}g^{\pm}}{\partial x^{i-l-1} \partial y^{j-1}}(x,0),\,\,  \text{ if }\, i>2k^{\pm},
	\end{aligned}
\]
with
\[
\begin{array}{l}
f^{\pm}(x)=\dfrac{\pm \delta Y^{\pm}(x,0)-a^{\pm}x^{2k^{\pm}-1} X^{\pm}(x,0)}{x^{2k^{\pm}} X^{\pm}(x,0)},\vspace{0.2cm}\\
g^{\pm}(x,y)=\dfrac{\pm X^{\pm}(x,0)Y^{\pm}(x,y) \mp X^{\pm}(x,y)Y^{\pm}(x,0)}{y \delta X^{\pm}(x,y)X^{\pm}(x,0)},
\end{array}
\]
and
\[
a^{\pm}=\dfrac{1}{(2k^{\pm}-1)!|X^{\pm}(0,0)|}\dfrac{\partial^{2k^{\pm}-1}Y^{\pm}}{\partial x^{2k^{\pm}-1}}(0,0).
\]

\bigskip

The coefficients  $\alpha_n$ of the series \eqref{eq:series} are given recursively by
\[
	\begin{cases}
		\alpha_1= -1,\vspace{0.1cm}\\
		\alpha_n=\dfrac{p^{\pm}_{n,k^{\pm}}(\alpha _1, \alpha _2, \cdots \alpha _{n-1}) - \mu^{\pm}_{n + 2k^{\pm} -1}}{2k^{\pm} \mu^{\pm}_{2k^{\pm}}}, 
	\end{cases}
\]
where 
\[
p^{\pm}_{n,k }\big(\alpha_1,\ldots,\alpha_{n-1}\big)= \mu^{\pm}_{2k } \hat B_{n+2k -1,2k}\big(\alpha_1,\ldots,\alpha_{n -1},0\big) +  \sum_{i=2k +1}^{n+2k -1} \mu^{\pm}_i \hat B_{n+2k -1,i}\big(\alpha_1,\ldots,\alpha_{n+2k -i}\big),
\]
and
\[
\mu_i^{\pm}= \dfrac{1}{i!}\sum_{j=1}^{i} (\mp \delta)^j {{i}\choose{j}} (y_j^\pm)^{(i-j)}(0).
\]

\bigskip

In the recurrences above, for $p$ and $q$ positive integers, $B_{p,q}$ and $\hat B_{p,q}$ denote, respectively, the {\it partial Bell polynomials} and  {\it ordinary Bell polynomials} (see, for instance \cite{comtet}):
\[
\begin{array}{l}
\displaystyle B_{p,q}(x_1,\ldots,x_{p-q+1})=\sum\dfrac{p!}{b_1!\,b_2!\cdots b_{p-q+1}!}\prod_{j=1}^{p-q+1}\left(\dfrac{x_j}{j!}\right)^{b_j} \text{ and }\vspace{0.2cm}\\
\displaystyle \hat B_{p,q}(x_1,\ldots,x_{p-q+1})=\sum\dfrac{p!}{b_1!\,b_2!\cdots b_{p-q+1}!}\prod_{j=1}^{p-q+1}x_j^{b_j},
\end{array}
\]
where the sums are performed  
on the $(p-q+1)$-tuples of nonnegative integers $(b_1,b_2,\cdots,b_{p-q+1})$ satisfying $b_1+2b_2+\cdots+(p-q+1)b_{p-q+1}=p,$ and
$b_1+b_2+\cdots+b_{p-q+1}=q.$ The partial Bell polynomials are implemented in algebraic manipulators as Mathematica and Maple, while the ordinary Bell polynomials can be computed as follows
\[
 \hat B_{p,q}(x_1,\ldots,x_{p-q+1})=\dfrac{q!}{p!}B_{p,q}(1!x_1,\ldots,(p-q+1)!x_{p-q+1}).
\]

Furthermore, in \cite[Appendix A]{novaessilva2020}, one can find implemented Mathematica algorithms for computing $y_i$ and $\alpha_n$.

\section*{Acknowledgments}

The authors thank the referee for the constructive comments and
suggestions which led to an improved version of the manuscript.

DDN is partially supported by S\~{a}o Paulo Research Foundation (FAPESP) grants 2021/10606-0, 2018/ 13481-0, and 2019/10269-3, and by Conselho Nacional de Desenvolvimento Cient\'{i}fico e Tecnol\'{o}gico (CNPq) grants 306649/2018-7,  438975/2018-9, and 309110/2021-1. LAS is partially supported by Coordena\c{c}\~{a}o de Aperfei\c{c}oamento de Pessoal de N\'{i}vel Superior (CAPES) grant 001.

\bibliographystyle{abbrv}
\bibliography{references}

\end{document}